\documentclass[twocolumn,numbook]{svjour3}
\usepackage{ifpdf}
\ifpdf
\usepackage[utf8]{inputenc}
\usepackage[T1]{fontenc}
\usepackage[all,pdf,2cell]{xy}\UseAllTwocells\SilentMatrices
\else
\usepackage[all,xdvi,2cell]{xy}\UseAllTwocells\SilentMatrices
\fi
\usepackage{amsmath}
\usepackage{amssymb}
\usepackage{cmll}
\usepackage{amsfonts}
\usepackage{enumerate}
\usepackage{natbib}
\usepackage{cleveref}
\usepackage{todonotes}
\usepackage[english]{babel}
\newcommand{\id}{\mathrm{id}}
\newcommand{\card}{\mathrm{card}}
\newcommand{\C}{\mathcal{C}}
\newcommand{\D}{\mathcal{D}}
\newcommand{\newcategory}[1]{\expandafter\newcommand\csname #1\endcsname{\mathbf{#1}}}
\newcategory{BPosInv}
\newcategory{BPos}
\newcategory{Pos}
\newcategory{EA}
\newcategory{OMP}
\newcategory{Set}

\usepackage{newunicodechar}
\newunicodechar{≤}{\leq}
\newunicodechar{∈}{\in}

\newcommand{\pperp}{\mathop{\Perp}}
\newcommand{\perpop}{\mathop{\perp}}
\begin{document}
\title{Orthomodular posets are algebras over bounded posets with involution}
\author{Gejza Jenča}
\institute{
G. Jenča \at
Department of Mathematics and Descriptive Geometry\\
Faculty of Civil Engineering,
Slovak University of Technology,
	Slovak Republic\\
              \texttt{gejza.jenca@stuba.sk}
}
\maketitle
\begin{abstract}
We prove that there is a monadic adjunction between the category of bounded posets
with involution and the category of orthomodular posets. 
\end{abstract}
\section{Introduction}
Effect algebras were introduced in~\cite{FouBen:EAaUQL} as (at that point in time) the most general version of quantum
logics. The motivating example was the set of all Hilbert space effects, a notion
that plays an important role in quantum mechanics \cite{Lud:FoQM,BusLahMit:TQToM}.
An equivalent definition was given independently
in~\cite{KopCho:DP}. Later it turned out that
both groups of authors rediscovered the definition given already
in~\cite{GiuGre:TaFLfUP}.

In \cite{kalmbach1977orthomodular} the following theorem was proved.
\begin{theorem}\label{thm:kalmbach}
Every bounded lattice $L$ can be embedded into an orthomodular lattice $K(L)$.
\end{theorem}
The proof of the theorem is constructive, $K(L)$ is known under the name {\em
Kalmbach extension} or {\em Kalmbach embedding}. In \cite{mayet1995classes},
authors proved that Theorem \ref{thm:kalmbach} can be generalized: every bounded
poset $P$ can be embedded in an orthomodular poset $K(P)$.  In fact,
by~\cite{harding2004remarks}, this $K$ is then left adjoint to the
forgetful functor from orthomodular posets to bounded posets. This adjunction gives
rise to a monad on the category of bounded posets, which we call the {\em Kalmbach
monad}.

For every monad $(T,\eta,\mu)$ on a category $\C$, there is an {\em
Eilenberg-Moore} category $\C^T$ (sometimes called the {\em category of algebras
over $T$}
or the {\em category of modules over $T$}). The category $\C^T$ comes equipped with
a canonical adjunction between $\C$ and $\C^T$ and this adjunction gives rise to
the original monad $T$ on $\C$. A functor equivalent to a right adjoint $\C^T\to\C$
is called {\em monadic}.

It was proved in~\cite{jenca2015effect} that the Eilenberg-Moore category for
the Kalmbach monad is isomorphic to the category of effect algebras. In other
words, the forgetful functor from the category of effect algebras to the category
of bounded posets is monadic. Later, it was proved 
in~\cite{jenca2020pseudo} that the forgetful functor from the category of
pseudo-effect algebras (see \cite{dvurecenskij2001pseudoeffect}) to the category of
bounded posets is monadic. Recently, it was proved
in~\cite{vandewetering2021categorical} that both $\omega$-complete effect algebras
and $\omega$-complete effect monoids can be represented as categories of algebras
over the category of bounded posets.

On the other hand, it clearly follows from  
\cite{jenca2015effect}  that the right adjoint functor from the category of
orthomodular posets to the category of bounded posets is \underline{not} monadic.
Indeed, this right adjoint gives rise to the Kalmbach monad and the Eilenberg-Moore
category for the Kalmbach monad is the category of effect algebras
which is clearly not equivalent to the category of orthomodular posets.

This leads to a natural question: is the category of orthomodular posets equivalent
to a category of algebras for some nontrivial monad? In the present
paper, we answer this question in the positive. We prove that the forgetful
functor from the category of orthomodular posets to the category of bounded posets
equipped with involution is monadic.

\section{Preliminaries}

We assume familiarity with basics of category theory, see
\cite{mac1998categories,awodey2006category,riehl2016category} for reference.

\subsection{Posets and bounded posets}
$\Pos$ is the usual category of posets equipped with isotone maps.
A {\em bounded poset} is a structure $(P,\leq,0,1)$ such that
$\leq$ is a partial order on $P$ and $0,1\in P$ are the bottom and top
elements of $(P,\leq)$, respectively.

Let $P_1,P_2$ be bounded posets. A map $f:P_1\to P_2$ is a 
{\em morphism of bounded posets} if and only if it satisfies the following
conditions.
\begin{itemize}
\item $f$ is isotone.
\item $f(1)=1$ and $f(0)=0$.
\end{itemize}

The category of bounded posets is denoted by $\BPos$.

\subsection{Bounded posets with involution}

An {\em involution} on a poset $P$ is mapping $'\colon P\to P$ satisfying the following
conditions.
\begin{itemize}
\item For all $x,y∈P$, $x≤y$ if and only if $y'≤x'$.
\item For all $x∈P$, $x''=x$.
\end{itemize}

Note that every bounded poset with involution $(P,≤,',0,1)$ satisfies $0'=1$.
Indeed, for every $x∈P$, $0≤x'$ implies $x''=x≤0'$, so $0'$ is the greatest element
of $P$.

Two elements $x,y$ of a bounded poset with involution are said to be {\em
orthogonal} (in symbols $x\perp y$), if $x\leq y'$.
Note that $\perpop$ is a symmetric relation and
that for all $x∈P$, $x\perp 0$.

The objects of the category $\BPosInv$ are bounded posets with
involution (sometimes called {\em involutive boun\-ded posets}).  A mapping $f\colon
P\to Q$ between two bounded posets with involution is a morphism in $\BPosInv$ if
and only if the following conditions are satisfied.
\begin{itemize}
\item $f$ is a morphism of bounded posets.
\item For all $x∈P$, $f(x')=(f(x))'$.
\end{itemize}

\subsection{Effect algebras}

An {\em effect algebra} 
is a partial algebra $(E,+,0,1)$ with a binary 
partial operation $+$ and two nullary operations $0,1$ satisfying
the following conditions.
\begin{enumerate}[(E1)]
\item If $a+b$ is defined, then $b+a$ is defined and
		$a+b=b+a$.
\item If $a+b$ and $(a+b)+c$ are defined, then
		$b+c$ and $a+(b+c)$ are defined and
		$(a+b)+c=a+(b+c)$.
\item For every $a∈E$ there is a unique $a'∈E$ such that
		$a+a'$ exists and $a+a'=1$.
\item If $a+1$ is defined, then $a=0$.
\end{enumerate}

In an effect algebra $E$, we write $a≤b$ if and only if there is
$c∈E$ such that $a+c=b$.  It is easy to check that for every effect
algebra $(E,≤,',0,1)$ is a bounded poset with involution.
Moreover, it is possible to introduce
a new partial operation $-$; $b-a$ is defined if and only if $a≤
b$ and then $a+(b-a)=b$.  It can be proved that, in an effect
algebra, $a+b$ is defined if and only if $a\perp b$.

Let $E_1$, $E_2$ be effect algebras. A map $f:E_1\to E_2$ is called a
{\em morphism of effect algebras} if and only if it satisfies the following conditions.
\begin{itemize}
\item $f(1)=1$.
\item If $a+b$ exists, then $f(a)+f(b)$ exists and $f(a+b)=f(a)+f(b)$.
\end{itemize}
The category of effect algebras is denoted by $\EA$.

We note that every morphism of effect algebras is isotone and preserves $0$, $1$
and the involution, hence every morphism of effect algebras is a morphism of the
underlying bounded posets with involution. Moreover, every morphism of effect
algebras preserves the partial operation $-$.

\subsection{Orthomodular posets}

An {\em orthomodular poset} is a bounded poset with involution $(A,≤,',0,1)$ 
satisfying the following conditions, for all $x,y\in A$.
\begin{enumerate}[(OMP1)]
\item $x\wedge x'=0$.
\item If $x\perp y$, then $x\vee y$ exists.
\item If $x≤y$, then $x\vee (x\vee y')'=y$.
\end{enumerate}

Let $A_1,A_2$ be orthomodular posets. A map $f:A_1\to A_2$ is called a {\em
morphism of orthomodular posets} if and only if it is a morphism of bounded
posets with involution such that for all $x\perp y$ in $A_1$,
$f(x)\perp f(y)$ and $f(x\vee y)=f(x)\vee f(y)$.

Let $A$ be an orthomodular poset. A subset $B\subseteq A$ is a {\em subalgebra}
of $A$ if and only if $0,1\in B$, $B$ is closed with respect to $'$ and
for all $x,y\in B$ such that $x\perp y$, $x\vee y\in B$.

The category of orthomodular posets is denoted by $\OMP$. If $A$ is an orthomodular
poset, then we may introduce a partial $+$ operation on $A$ by the following rule:
$x+y$ exists if and only if $x\perp y$ and then $x+y:=x\vee y$.  The resulting structure is
then an effect algebra. This gives us the object part of an evident full and
faithful functor $\OMP\to\EA$.

\begin{proposition}\label{prop:OMPiscomplete}
The category $\OMP$ is small-complete.
\end{proposition}
\begin{proof}
It is easy to check that a product of every family of orthomodular posets can be
constructed as a product of underlying bounded posets, the involution $'$ is
defined pointwise. For a parallel pair of morphisms $f,g\colon A\to B$ in $\OMP$,
their equalizer is the inclusion of a subalgebra $E=\{x\in A:f(x)=g(x)\}$ into $A$.
Since the category $\OMP$ has all products and all equalizers, it has all small
limits.
\end{proof}

The following functors will be used in what follows.

\subsection{The functor $U\colon\BPosInv\to\Pos$}

$U$ is just the forgetful functor that takes every object of
$\BPosInv$ to its underlying poset.

\subsection{The functor $\perpop\colon\BPosInv\to\Pos$}

The orthogonality relation on the objects of $\BPosInv$ can
be exhibited as a functor from $\BPos$ to $\Pos$, as follows.
For every bounded poset with involution $P$, $\perpop(P)$ is the
order ideal in the poset $U(P)\times U(P)$ given by the rule
$$
\perpop(P)=\{(x,y)∈U(P)\times U(P)\mid x≤y'\}.
$$
For a morphism $f\colon P\to Q$ in $\BPosInv$, $\perpop(f)\colon
\perpop(P)\to\perpop(Q)$ is given by $\perpop(f)(x,y)=(f(x),f(y))$.  
\subsection{The functor $\pperp\colon\BPosInv\to\Pos$}

For every bounded poset with involution $P$, $\pperp(P)$ is the
order ideal in the poset $U(P)\times U(P)\times U(P)$ given by the rule
\begin{multline*}
\pperp(P)=\{(x,y,z)∈U(P)\times U(P)\times U(P)\mid \\
x\perp y\text{ and }x\perp
z\text{ and }y\perp z\}.
\end{multline*}

On morphisms, $\pperp$ is defined coordinatewise, similarly as we did for $\perpop$.

\subsection{The functor $I\colon\Pos\to\Pos$}

For every poset $P$, let us write $I(P)$ for the set of comparable pairs $\{(a,b)∈P\times P\colon a≤b\}$ and partially
order $I(P)$ by the rule
$(a,b)≤(c,d)$ if and only if $c≤a≤b≤d$. Note that the elements of
$I(P)$ can be identified with closed intervals of $P$, ordered by inclusion.
We shall write $[a≤b]$ or $[b\geq a]$ for the element $(a,b)$ of $I(P)$.
The construction $P\mapsto I(P)$ can be made into a functor $\Pos\to\Pos$ by the rule
$I(f)([a≤b])=[f(a)≤f(b)]$.

\subsection{General adjoint functor theorem}

Adjoint functor theorems give conditions under which a continuous functor $G$
has a left adjoint $F$.  This allows us to avoid construction of the functor
$F$, which is sometimes a difficult endeavor.

\begin{theorem}
\cite[Theorem~V.6.2]{mac1998categories} \cite{freyd1964abelian}
\label{thm:gaft}
Given a locally small, small-complete category $\D$, a functor
$G\colon\D\to\C$ is a right adjoint if and only if $G$ preserves
small limits and satisfies the following 

\underline{Solution Set Condition:} for each object $X$ of $\C$ there is a set $I$
and an $I$-indexed family of arrows $h_i\colon X\to G(A_i)$ such that
every arrow $h\colon X\to G(A)$ can be written as a composite
$h=G(j)\circ h_i$ for some $j\colon A_j\to A$. 
\end{theorem}

\subsection{Beck's monadicity theorem}

A functor $G\colon\D\to\C$ is {\em monadic} if and only if it is equivalent
to the forgetful functor from the category of algebras $\C^T$ to $\C$ for a monad $T$ on $\C$.
A colimit (or a limit) in a category $\C$ is {\em absolute} if and only if it is preserved by every
functor with domain $\C$.

\begin{theorem}\cite[Theorem~VI.7.1]{mac1998categories} \cite{beck1967triples}
\label{thm:beckmonadicity}
A functor 
$$
G\colon\D\to\C
$$ 
is monadic if and only if $G$ is a right adjoint
and $G$ creates coequalizers for those parallel pairs $f,g\colon A\to B$
in $\D$, for which
$$
\xymatrix{
G(A)
	\ar@<.5ex>[r]^{G(f)}
	\ar@<-.5ex>[r]_{G(g)}
&
G(B)
}
$$
has an absolute coequalizer in $\C$.
\end{theorem}
Beck's monadicity theorem is a device that allows us to prove that a functor is monadic without
having to explicitly describe the monad $T$ on $\C$ arising from the adjunction, 
describe its category of algebras $\C^T$ and to prove that $\C^T$ is equivalent to $\D$.

\section{The result}

The aim of this paper is to prove that the obviously defined forgetful functor
from $\OMP$ to $\BPosInv$ is a monadic right adjoint. To do this, we need to
characterize orthomodular posets as bounded posets with involution equipped with an
additional structure.

\begin{proposition}
\label{prop:OMP}
Let $A$ be a bounded poset with involution equipped with two
partial operations $+$ and $-$ such that
\begin{itemize}
\item $x+y$ is defined if and only if $x\perp y$ and
\item $y-x$ is defined if and only if $x\leq y$.
\end{itemize}
Suppose that the mappings 
$$
+\colon\perpop(A)\to U(A)
$$ 
and 
$$
-\colon IU(A)\to U(A)
$$
corresponding
to these partial operations are isotone 
and that the following conditions are satisfied.
\begin{enumerate}[({A}1)]
\setcounter{enumi}{-1}
\item For all $a∈A$, $a\perp 0$ and $a+0=a$.
\item For all $a,b∈A$, $a\perp b$ implies that $b\perp a$ and then $a+b=b+a$
\item For all $a,b,c∈A$, if $a\perp b$ and $b\perp c$ and $a\perp c$, then $a\perp (b+c)$ and
$(a+b)\perp c$ and $a+(b+c)=(a+b)+c$.
\item For all $a,b∈A$ such that $a\perp b$, $a+b\geq a$ and $(a+b)-a=b$.
\item For all $a∈A$, $a+a'=1$.
\end{enumerate}
Then $A$ is an orthomodular poset with $x\vee y=x+y$, for all $x\perp y$.

Moreover, for every orthomodular poset $A$, putting $x+y=x\vee y$ and
$x-y=x\wedge y'$ gives us a pair of partial operations satisfying these conditions.
\end{proposition}
\begin{proof}
The fact that every orthomodular poset satisfies these
conditions is well known.
By \cite[Theorem 5.3]{FouBen:EAaUQL}, an effect algebra is an orthomodular poset
with $a\vee b=a+b$ if and only if it satisfies (A2). So it remains to prove that
$A$ is an effect algebra. The conditions (E1) and (A1) are the same.
To prove (E2), suppose that $a,b,c∈A$ are such that $a\perp b$ and
$(a+b)\perp c$, that means $a+b≤c'$. From (A0) and the fact that $+$ is isotone
we obtain $a=(a+0)≤(a+b)≤c'$, hence $a\perp c$ and (similarly) $b\perp c$. By (A3),
$a\perp (b+c)$ and $(a+b)+c=a+(b+c)$. The existence part of (E3) is exactly (A4). To
prove the uniqueness part of (E3), suppose that $1=a+b_1=a+b_2$. By (A3), it then follows that
$$
b_1=(a+b_1)-a=1-a=(a+b_2)-a=b_2.
$$
To prove (E4), note that $a+1$ defined means that $a\perp 1$, which means $a≤1'=0$
and hence $a=0$.
\end{proof} 

Since every orthomodular poset is a bounded poset with involution and we defined
a morphism of orthomodular poset as a special type of morphism of the underlying
bounded posets with involution, there is an obviously defined forgetful functor
$G\colon\OMP\to\BPosInv$.

The main result of the present paper follows.
\begin{theorem}
The forgetful functor 
$$
G\colon\OMP\to\BPosInv
$$ is monadic.
\end{theorem}
\begin{proof}
Let us apply Theorem~\ref{thm:gaft} to prove that $G$ is a right adjoint functor.
By Proposition~\ref{prop:OMPiscomplete},
$\OMP$ is small-complete.
It is easy to check that $G$ preserves all small limits.
Let us check the Solution
Set Condition. Let $P$ be a bounded poset with involution. Let $\mathcal W_P$ be a
set of bounded posets with involution
such that for every bounded poset with involution $V$ with 
$\card(P)\leq\card(V)\leq \max(\card(P),\aleph_0)$, there is a $W\in\mathcal W_P$ such that
$W$ is isomorphic to $V$. Consider
the family $\mathcal H_P=\{h_i\}_{i\in I}$ of all $\BPosInv$-morphisms $h_i\colon
P\to G(A_i)$, where $A_i$ is an 
orthomodular poset and $G(A_i)\in\mathcal W_P$. For every $\BPosInv$-morphism
$h\colon P\to G(A)$, the cardinality of the subalgebra $B$ of $A$ that is generated by the range of $h$
is bounded below by $\card(P)$ and above by $\max(\card(P),\aleph_0)$. Write $j\colon B\to A$
for the embedding of the subalgebra $B$ into $A$. Clearly, $h=G(j)\circ h_i$ for some
$h_i\in \mathcal H_P$. Since $G$ preserves small limits and the Solution Set Condition is satisfied,
$G$ is a right adjoint.

We have proved that $G$ is a right adjoint, so we may apply Theorem~\ref{thm:beckmonadicity}.
Let $A,B$ be orthomodular posets, let $f,g\colon A\to B$ be morphisms of
orthomodular posets. Suppose that
\begin{equation}
\label{diag:absolute}
\xymatrix{
G(A)
	\ar@<.5ex>[r]^{G(f)}
	\ar@<-.5ex>[r]_{G(g)}
&
G(B)
	\ar[r]^q
&
Q
}
\end{equation}
is an absolute coequalizer. Assuming this, we need to prove that
there is a unique morphism of orthomodular posets $\widehat{q}\colon B\to \widehat{Q}$ such that
\begin{equation}
\xymatrix{
A
	\ar@<.5ex>[r]^{f}
	\ar@<-.5ex>[r]_{g}
&
B
	\ar[r]^{\widehat{q}}
&
\widehat{Q}
}
\end{equation}
is a coequalizer in $\OMP$ and $Q=G(\widehat{Q})$, $q=G(\widehat{q})$. Let us 
prove that such $\widehat{q}$ exists. We use the fact that \eqref{diag:absolute}
is an absolute coequalizer to equip the bounded poset with involution $Q$ with a
structure of an orthomodular poset, in the sense of \Cref{prop:OMP}. 
Then we prove that $q$ comes from 
a morphism of orthomodular posets. Finally, we prove that this morphism of
orthomodular posets is a coequalizer of $f,g$ in $\OMP$.

Let $A$ be an orthomodular poset. Consider \Cref{prop:OMP}; the partial operation $+_A$ on $A$ is defined for
pairs $a,b\in A$ with $a\perp b$, so we may represent it as an isotone mapping (in
other words, a morphism in $\Pos$) from the poset $\perpop G(A)$ to the poset
$UG(A)$. Similarly, the partial operation $-_A$ on $A$ can be represented by an
$\Pos$-morphism $-_A\colon IUG(A)\to UG(A)$.  Moreover, every morphism $h\colon
X\to Y$ of orthomodular posets preserves the partial operations $+$ and $-$, that
means that the diagrams
\begin{align}
\label{diag:plusnatural}
\xymatrix{
\perpop G(X)
	\ar[r]^{\perpop G(h)}
	\ar[d]_{+_X}
&
\perpop G(Y)
	\ar[d]^{+_Y}
\\
UG(X)
	\ar[r]_{UG(h)}
&
UG(Y)
}
\\
\label{diag:minusnatural}
\xymatrix{
IUG(X)
	\ar[r]^{IUG(h)}
	\ar[d]_{-_X}
&
IUG(Y)
	\ar[d]^{-_Y}
\\
UG(X)
	\ar[r]_{UG(h)}
&
UG(Y)
}
\end{align}
commute. Hence $+$ and $-$ are natural transformations in the category of functors
$[\OMP,\Pos]$.

Consider the diagram
\begin{equation}
\label{diag:defineplus}
\xymatrix@C=3pc{
\perpop G(A)
	\ar@<.5ex>[r]^{\perpop G(f)}
	\ar@<-.5ex>[r]_{\perpop G(g)}
	\ar[d]_{+_A}
&
\perpop G(B)
	\ar[r]^{\perpop (q)}
	\ar[d]^{+_B}
&
\perpop (Q)
	\ar@{.>}[d]^{\boxplus}
\\
UG(A)
	\ar@<.5ex>[r]^{UG(f)}
	\ar@<-.5ex>[r]_{UG(g)}
&
UG(B)
	\ar[r]^{U(q)}
&
U(Q)
}
\end{equation}
Since $f$ and $g$ are morphisms of orthomodular posets,
the naturality of $+$ implies that both $f$ and $g$
left-hand squares in \eqref{diag:defineplus} commute.
From this and from the fact that the bottom row is
a coequalizer, it follows that the morphism $+_B\circ U(q)$
coequalizes the top pair of parallel arrows.
Indeed,
\begin{equation*}
\begin{split}
U(q)\circ+_B\circ \perpop G(f)&=U(q)\circ UG(f)\circ +_A\\&=U(q)\circ UG(g)\circ +_A\\&=
U(q)\circ+_B\circ \perpop G(g).
\end{split}
\end{equation*}
Since \eqref{diag:absolute} is an absolute coequalizer,
the top row in \eqref{diag:defineplus} is a coequalizer.
Since the top row in \eqref{diag:defineplus} is a coequalizer, there is a unique
arrow $\boxplus\colon \perpop (Q)\to U(Q)$ making the right-hand square of
\eqref{diag:defineplus} commute.
This way, we equipped the involutive bounded poset $Q$ with a partial binary operation
$\boxplus$, defined for all orthogonal pairs of elements of $Q$.

In an analogous way, we may the use the diagram
\begin{equation}
\label{diag:defineminus}
\xymatrix@C=3pc{
IUG(A)
	\ar@<.5ex>[r]^{IUG(f)}
	\ar@<-.5ex>[r]_{IUG(g)}
	\ar[d]_{-_A}
&
IUG(B)
	\ar[r]^{IU(q)}
	\ar[d]^{-_B}
&
IU(Q)
	\ar@{.>}[d]^{\boxminus}
\\
UG(A)
	\ar@<.5ex>[r]^{UG(f)}
	\ar@<-.5ex>[r]_{UG(g)}
&
UG(B)
	\ar[r]^{U(q)}
&
U(Q)
}
\end{equation}
to define a partial operation $\boxminus\colon IU(Q)\to U(Q)$ on $Q$.

Let us prove that these partial operations on $Q$ satisfy the conditions of
\Cref{prop:OMP}. We will proceed as follows: we will show that the fact
that $X$ is an orthomodular
poset can be expressed by encoding the conditions (A0)--(A4) by means of 
functors and natural transformations. For (A0), we will then give a detailed
proof that $Q$ satisfies (A0). After that, we will observe that the proof for
the remaining conditions can be given in a similar way.

\begin{enumerate}[({A}1)]
\setcounter{enumi}{-1}
\item For every bounded poset $P$ with involution,
there is an isotone mapping 
$$
(0,\_)_P\colon U(P)\to \perpop(P)
$$ that maps
every $a∈U(P)$ to the pair $(0,a)\in\perpop(P)$. This $ob(\BPosInv)$ indexed
family of arrows forms a natural transformation from $U$ to $\perpop$. The (A0)
property means that for every orthomodular poset $X$, the diagram
\begin{equation}
\label{diag:A0}
\xymatrix@C=3.5pc{
UG(X)
	\ar[r]^-{(0,\_)_{G(X)}}
	\ar[rd]_{\id}
&
\perpop G(X)
	\ar[d]^{+_X}
\\
~
&
UG(X)
}
\end{equation}
commutes. Therefore, $+\circ((0,\_)G)=\id_{UG}$ in the category 
of functors $[\OMP,\Pos]$. Equivalently, the diagram 
\begin{equation}
\label{diag:A0nats}
\xymatrix@C=3.5pc{
UG
	\ar[r]^-{(0,\_)G}
	\ar[rd]_{\id}
&
\perpop G
	\ar[d]^{+}
\\
~
&
UG
}
\end{equation}
in the category $[\OMP,\Pos]$ commutes.
\item
For every bounded poset with involution $P$, there is a poset automorphism 
$s_P\colon \perpop(P)\to\perpop(P)$ given by the rule $s_P(a,b)=(b,a)$. Moreover,
the family of morphisms $s_\_$ indexed by the objects of $\BPosInv$ forms a
natural transformation 
$$
s\colon\perpop\to\perpop
$$ 
in the category of functors $[\OMP,\Pos]$.

The axiom
(A1) means, that for every orthomodular poset $X$, the diagram
\begin{equation}
\xymatrix@C=2.5pc{
\perpop G(X)
	\ar[r]^-{s_{G(X)}}
	\ar[rd]_{+_X}
&
\perpop G(X)
	\ar[d]^{+_X}
\\
~
&
UG(X)
}
\end{equation}
commutes. This implies that $+\circ (sG)=+$ in the category of functors
$[\OMP,\Pos]$.
\item
For every orthomodular poset $X$, there are isotone mappings
$$
r_X,l_X\colon\pperp G(X)\to\perpop G(X)
$$
given by the rules
$$
r_X(a,b,c)=(a,b+c)
\qquad
l_X(a,b,c)=(a+b,c)
$$
The families of $\Pos$-morphisms $r_\_$ and $l_\_$ indexed by objects of $\OMP$
form a pair of 
natural transformations 
$$
r,l\colon\pperp G\to\perpop G
$$ 
in the category of functors
$[\OMP,\Pos]$. The property (A3) then means that for every orthomodular poset
$X$, the diagram
\begin{equation}
\label{diag:A1}
\xymatrix{
\pperp G(X)
	\ar[r]^{r_X}
	\ar[d]_{l_X}
&
\perpop G(X)
	\ar[d]^{+_X}
\\
\perpop G(X)
	\ar[r]_{+_X}
&
UG(X)
}
\end{equation}
commutes, so $+\circ r=+\circ l$ in the category of functors $[\OMP,\Pos]$.
\item
For every orthomodular poset $X$, there is a mapping
$p_X\colon\perpop G(X)\to IUG(X)$ given by the rule
$$
p_X(a,b)=(a+b,a)
$$
Let $\pi_X\colon\perpop G(X)\to UG(X)$ be the projection 
to the first entry, given by the rule $\pi(a,b)=a$.
Again, the families $p_\_$ and $\pi_\_$ indexed by objects of
$\OMP$ are natural transformations in the category of functors
$[\OMP,\Pos]$.

By (A3), the diagram
$$
\xymatrix{
\perpop G(X)
	\ar[r]^{p_X}
	\ar[rd]_{\pi_X}
&
IUG(X)
	\ar[d]^{-_X}
\\
~
&
UG(X)
}
$$
commutes, meaning that $-\circ p=\pi$ in the category $[\OMP,\Pos]$.
\item
To express this property using functors and natural transformations, we need to
change the target category of our functors from $\Pos$ to $\Set$. This appears to
necessary because
the mapping $a\mapsto (a,a')$ is not isotone.

Write $S\colon\Pos\to\Set$ for the straightforwardly defined 
functor that takes every poset to its underlying set. For every orthomodular poset
$X$, there is an morphism of sets $c_X\colon SUG(X)\to S\perpop G(X)$ given
by the rule $c_X(a)=(a,a')$. The property (A4) then means that
the diagram
$$
\xymatrix{
SUG(X)
	\ar[r]^{c_X}
	\ar[rd]_{1}
&
S\perpop G(X)
	\ar[d]^{S+_X}
\\
~
&
SUG(X)
}
$$
commutes; here, $1\colon SUG(X)\to SG(X)$
denotes the constant mapping with value $1∈SG(X)$.
Therefore, $c\circ (S+)=1$ in the category of functors $[\OMP,\Set]$.
\end{enumerate}

Let us prove in detail that the partial operation $\boxplus$ on $Q$ satisfies
(A0). Consider the diagram
\begin{equation}
\label{diag:A0transfer1}
\xymatrix@C=2.5pc{
UG(A)
	\ar@<.5ex>[r]^{UG(f)}
	\ar@<-.5ex>[r]_{UG(g)}
	\ar[d]_{(0,\_)_{G(A)}}
&
UG(B)
	\ar[r]^{U(q)}
	\ar[d]^{(0,\_)_{G(B)}}
&
U(Q)
	\ar[d]^{(0,\_)_Q}
\\
\perpop G(A)
	\ar@<.5ex>[r]^{\perpop G(f)}
	\ar@<-.5ex>[r]_{\perpop G(g)}
	\ar[d]_{+_A}
&
\perpop G(B)
	\ar[r]^{IU(q)}
	\ar[d]^{+_B}
&
\perpop(Q)
	\ar[d]^{\boxplus}
\\
UG(A)
	\ar@<.5ex>[r]^{UG(f)}
	\ar@<-.5ex>[r]_{UG(g)}
&
UG(B)
	\ar[r]^{U(q)}
&
U(Q)
}
\end{equation}
Since \eqref{diag:A0} commutes, both left and middle verticals of
\eqref{diag:A0transfer1} compose to identity, therefore
\begin{equation}
\label{diag:A0transfer2}
\xymatrix@C=3pc{
UG(A)
	\ar@<.5ex>[r]^{UG(f)}
	\ar@<-.5ex>[r]_{UG(g)}
	\ar[d]_{\id}
&
UG(B)
	\ar[r]^{U(q)}
	\ar[d]^{\id}
&
U(Q)
	\ar[d]^{\boxplus\circ(0,\_)_Q}
\\
UG(A)
	\ar@<.5ex>[r]^{UG(f)}
	\ar@<-.5ex>[r]_{UG(g)}
&
UG(B)
	\ar[r]^{U(q)}
&
U(Q)
}
\end{equation}
commutes. Note that if we replace the rightmost arrow in \eqref{diag:A0transfer2}
by $\id$, the diagram still commutes. However, by an analogous argument that
we used to define $+$ on $Q$, the rightmost vertical arrow in
\eqref{diag:A0transfer2} is unique, meaning that the diagram
\begin{equation}
\label{diag:A0Q}
\xymatrix@C=3.5pc{
U(Q)
	\ar[r]^-{(0,\_)_{Q}}
	\ar[rd]_{\id}
&
\perpop (Q)
	\ar[d]^{\boxplus}
\\
~
&
U(Q)
}
\end{equation}
commutes.
Therefore, $\boxplus\circ(0,\_)_Q=id_{U(Q)}$ or, in other
words, for all $x∈Q$, $0\boxplus x$ is defined and $0\boxplus x=x$.
Thus, the partial operation $\boxplus$ on $Q$ satisfies the condition (A0).

The reader should now observe that the essence of our proof that $\boxplus$
satisfies the contition (A0) lies in the fact that we can represent the condition
(A0) by the commutative diagram \eqref{diag:A0nats} containing functors of the type
$
\_\circ G\colon\OMP\to\Pos
$ and natural transformations among them.  A routine
argument then allows us to use the fact that the orthomodular posets $A$ and $B$
satisfy (A0) to show that \eqref{diag:A0Q} commutes, hence the partial operation
$\boxplus$ on $Q$ satisfies the condition (A0).  In the previous part of the proof
we have demonstrated that it is possible to express the remaining conditions
(A1)--(A4) by a commutative diagram in $[\OMP,\Pos]$ or $[\OMP,\Set]$. We may thus
prove in an analogous way that the partial operations $\boxplus$ and $\boxminus$ on
$Q$ satisfy the conditions (A1)--(A4) so we skip the this part of the proof.

We know that the partial operations $\boxplus$ and $\boxminus$ we defined on $Q$
satisfy the conditions (A0)--(A4).  In other words, there is a orthomodular poset
$\widehat{Q}$ such that $Q=G(\widehat{Q})$.  Moreover, the $\BPosInv$-morphism $q\colon
G(B)\to Q=G(\widehat{Q})$ satisfies $q(x+y)=q(x)\boxplus q(y)$, for all $x\perp y$
in $B$, because the right-hand square
of~\eqref{diag:defineplus} commutes. That means, there is a
morphism of orthomodular posets $\widehat q\colon B\to \widehat{Q}$ such that
$q=G(\widehat q)$.  With this
fact in mind, we may now observe that the diagram~\eqref{diag:defineplus} means
that $\widehat{q}\circ f=\widehat{q}\circ g$ in $\OMP$ and since the orthomodular
poset structure on $Q$ arising from the diagrams \eqref{diag:defineplus} and
\eqref{diag:defineminus} is unique, we see that $\widehat{Q}$ is unique.
Uniqueness of the morphism $\widehat{q}$ follows from the fact that $G$ is a
faithful functor.

Let us prove that $\widehat{q}$ is a coequalizer of the pair $f,g$ in $\OMP$.
Let $h\colon B\to C$ be a morphism of orthomodular posets such that $h\circ f=h\circ g$. Since
the diagram~\eqref{diag:absolute} is a coequalizer in $\BPosInv$, there is a unique morphism of
bounded posets with involution
$
e\colon G(\widehat{Q})\to G(C)
$ such that $e\circ G(\widehat{q})=e\circ q=G(h)$. It remains to prove that this $e$ preserves
the partial operation $\boxplus$ on $Q$. Consider the following diagram:
\begin{equation}
\label{diag:final}
\xymatrix@R=1.2pc{
\perpop G(B)
	\ar[rd]_{\perpop (q)}
	\ar[ddd]_{+_B}
	\ar[rr]^{\perpop G(h)}
&
~
&
\perpop G(C)
	\ar[ddd]^{+_C}
\\
~&
\perpop (Q)
	\ar[d]^{\boxplus}
	\ar[ru]_{\perpop (e)}
&
~
\\
~
&
U(Q)
	\ar[rd]^{U(e)}
&
~
\\
UG(B)
	\ar[ru]^{U(q)}
	\ar[rr]_{UG(h)}
&
~
&
UG(C)
}
\end{equation}

We need to prove that the right-hand square in~\eqref{diag:final} commutes. By the commutativity
of the diagram~\eqref{diag:defineplus}, we already know that
the left hand square in~\eqref{diag:final} commutes. 
moreover, as $G(h)=e\circ G(\widehat{q})$ we see that
$$
\perpop G(h)=\perpop(e\circ G(\widehat q))=\perpop(e)\circ\perpop G(\widehat q)=
\perpop(e)\circ \perpop(q)
$$
so the top triangle in \eqref{diag:final} commutes.
Similarly, we can prove
$$
UG(h)=U(e)\circ U(q)
$$
so the bottom triangle in \eqref{diag:final} commutes.
Since $h$ is a morphism of orthomodular posets, the outer square of~\eqref{diag:final} commutes. Therefore,
\begin{multline}
\label{eq:final}
+_C\circ\perpop(e)\circ\perpop(q)=
+_C\circ\perpop G(h)=\\
UG(h)\circ+_B=
U(e)\circ U(q)\circ+_B=
U(e)\circ\boxplus\circ\perpop(q)
\end{multline}

Since the top row in~\eqref{diag:defineplus} is a coequalizer, $\perpop(q)$ 
is a coequalizing arrow and thus an epimorphism. Therefore,  \eqref{eq:final} 
implies that $+_C\circ
\perpop(e)=U(e)\circ \boxplus$
and we see that the right-hand square of~\eqref{diag:final} commutes.
\end{proof}
\section{Further research}
Sice $G$ is a right adjoint, there is a functor 
$$
F\colon\BPosInv\to\OMP
$$ left adoint to $G$.
It would be interesting to describe this functor, in an explicit way.
\section*{Declarations}
\begin{description}
\item[\bf Funding:]This research is supported by grants VEGA 2/0142/20 and
1/0006/19, Slovakia and by the Slovak Research and Development Agency under the
contracts APVV-18-0052 and APVV-20-0069.
\item[\bf Conflict of interest/Competing interest:]None.
\item[\bf Availability of data and material:]Not applicable.
\item[\bf Code availability:]Not applicable.
\item[\bf Authors' contributions:]There is a single author. 
\end{description}

\end{document}